\documentclass[envcountsame,envcountchap]{article}

\newcommand{\revised}{March 18, 2021} 
\usepackage[matrix,arrow,ps]{xy}

\usepackage{graphicx}
\usepackage{subfigure}
\usepackage[usenames]{color}
\usepackage{amssymb,amsmath,amsthm}
\usepackage{amsfonts}
\usepackage{graphicx}
\usepackage[section]{placeins}
\usepackage{float}

\newtheorem{theorem}{Theorem}[section]
\newtheorem{lemma}[theorem]{Lemma}

\newtheorem{definition}[theorem]{Definition}
\newtheorem{example}[theorem]{Example}
\newtheorem{remark}[theorem]{Remark}

\newcommand{\Vol}{{\operatorname{Vol}}}
\renewcommand{\d}{\partial}

\newcommand{\acal}{\mathcal{A}}

\newcommand{\ccal}{\mathcal{C}}

\newcommand{\diag}{{\operatorname{diag}}}
\newcommand{\dist}{{\operatorname{dist}}}
\newcommand{\vol}{{\operatorname{Vol}}}

\newcommand{\inner}{{\operatorname{inner}}}
\newcommand{\Outer }{{\operatorname{outer}}}
\newcommand{\John}{{\operatorname{John}}}

\newcommand{\xx}{{\bf x}}

\newcommand{\yy}{{\bf y}}
\newcommand{\zz}{{\bf z}}

\newcommand{\uu}{{\bf u}}

\newcommand{\ee}{{\bf e}}
\newcommand{\vv}{{\bf v}}
\newcommand{\ww}{{\bf w}}

\newcommand{\nn}{{\bf n}}

\newcommand{\IB}{\mathbb{B}}

\newcommand{\IS}{\mathbb{S}}

\newcommand{\II}{\mathbb{I}}
\newcommand{\IR}{\mathbb{R}}

\newcommand{\bea}{\begin{eqnarray*}}
\newcommand{\eea}{\end{eqnarray*}}

\newcommand{\beq}{\begin{equation}}
\newcommand{\eeq}{\end{equation}}

\newcommand{\ga}{\gamma}

\newcommand{\bz}{{\bf 0}}
\newcommand{\sphere}{\mathbb{S}^{N-1}}

\begin{document}
\date{}

\ {}\vskip .5 true in

\centerline{\Large\bf Closed-Form Parametric Equation for the}
\centerline{\Large\bf Minkowski Sum of $m$ Ellipsoids in $\IR^N$}
\centerline{\Large\bf and Associated Volume Bounds}

\vskip .3 true in

\centerline{Gregory S. Chirikjian}
\centerline{Department of Mechanical Engineering}
\centerline{National University of Singapore}
\centerline{mpegre@nus.edu.sg}
\centerline{and}
\centerline{Department of Mechanical Engineering}
\centerline{Johns Hopkins University}
\centerline{gchirik1@jhu.edu}

\vskip 0.1 true in

\centerline{Bernard Shiffman}
\centerline{Department of Mathematics}
\centerline{Johns Hopkins University}
\centerline{bshiffman@jhu.edu}

\vskip 0.1 true in

\centerline{\revised}

\vskip 0.1 true in

\begin{abstract} General results on convex bodies are reviewed and used to derive an
exact closed-form parametric formula for the Minkowski sum boundary of $m$ arbitrary ellipsoids in $N$-dimensional Euclidean space.
Expressions for the principal curvatures of these Minkowski sums are also derived.
These results are then used to obtain upper and lower volume bounds for the Minkowski sum of ellipsoids in terms of their defining matrices; the lower bounds are sharper than the
Brunn-Minkowski inequality.  A  reverse isometric inequality   for convex bodies  is also given.
\end{abstract}

\section{Introduction}\label{intro}

The concept of the Minkowski sum of two bodies in $N$-dimensional Euclidean space $\IR^N$ is fundamental in
the field of convex geometry. When $N=3$, Minkowski sums play important roles in applied fields such as robot motion planning, computational chemistry, and computer graphics { \cite{Chernousko, cs,Durieu,Kurzhanski1,Kurzhanski2,ruan2019,SLC,yy}}.

By a body we mean a bounded, connected, open (nonempty) subset of $\IR^N$. (Note that our definition of a body differs from the literature where  bodies are assumed to be compact rather than open and bounded.)  Given two such bodies, $B_1$ and $B_2$, their Minkowski sum is defined by
\begin{equation}
B_1 + B_2 \,:=\, \{{\bf x} + {\bf y} \,|\, {\bf x} \in B_1\,,\, {\bf y} \in B_2\}\,.
\label{minksumdef}
\end{equation}

Among all convex bodies with differentiable boundaries, `solid ellipsoids' of the form  $E=\{S\vv:\|\vv\|<1\}$, where $S$ is a nonsingular $N\times N$ matrix,
constitute a fairly broad, yet simple,
class of objects. Solid ellipsoids are convenient since their boundaries have both closed form parametric and implicit descriptions.
The ellipsoid  $\partial E$   can be parameterized as
\begin{equation}
{\bf x}(\phi) = S {\bf u}(\phi) \,\in\, \IR^N\,,
\label{paraS}
\end{equation}
where $\uu(\phi) \in \mathbb{S}^{N-1}$ (the unit sphere) is a unit vector, and $\phi = (\phi_1, ..., \phi_{N-1})$ are spherical angles (or any local coordinates) that parameterize the
sphere.

The corresponding implicit equation of $\partial E$ is
\begin{equation}
\Psi({\bf x}) \,=\, 1 \,\,\, {\rm where} \,\,\,
\Psi({\bf x}) \, :=\,  {\bf x}^T S^{-T}S^{-1} {\bf x}={\bf x}^T A^{-2} {\bf x} \,, \quad A := (SS^T)^{1/2}.
\label{impl}
\end{equation}
We note that $A$ is a positive-definite symmetric matrix, and the ellipsoid $\d E$ can be reparameterized as
\begin{equation}
{\bf x}(\uu(\phi)) = A {\bf u}(\phi)\,.
\label{para1}
\end{equation}


When using (\ref{para1}),
the unit normal ${\bf n}(\phi)$ to $\d E$ at ${\bf x}(\uu(\phi))$ is computed as
\begin{equation}
{\bf n}(\phi) \,=\, \left.\frac{(\nabla \Psi)({\bf x})}{\|(\nabla \Psi)({\bf x})\|}\right|_{{\bf x} = \xx(\uu(\phi))}\,=\, \frac{A^{-2} \xx(\uu(\phi))}{\|A^{-2} \xx(\uu(\phi))\|}
\,=\, \frac{A^{-1} {\bf u}(\phi)}{\|A^{-1} {\bf u}(\phi)\|}\,.
\label{n-from-u}
\end{equation}
Since this means that $A {\bf n}(\phi)$ is a scalar multiple of ${\bf u}(\phi)$, it is possible to invert the above expression  as
\begin{equation}
{\bf u}(\phi) \,=\, \frac{A {\bf n}(\phi)}{\|A {\bf n}(\phi)\|}\,.
\label{u-from-n}
\end{equation}

Combining \eqref{para1} and \eqref{u-from-n},
\begin{equation}
{\bf x}_{\partial E}(\nn(\phi))\, :=\,\xx(\uu(\phi))\,=\,  \frac{A^2 {\bf n}(\phi)}{\|A {\bf n}(\phi)\|}\,,\label{x-from-n}\eeq
which confirms the bijectivity of the Gauss map $\xx\mapsto\nn$ for ellipsoids. (The Gauss map from the boundary of a convex body   is always bijective onto ${\mathbb S}^{N-1}$.)
When using \eqref{paraS} with a nonsymmetric matrix $S$, the  equation corresponding to \eqref{n-from-u} is
\begin{equation}
{\bf n}(\phi) \,=\,   \frac{S^{-T} S^{-1} \xx(\uu(\phi))}{\|S^{-T} S^{-1} \xx(\uu(\phi))\|}
\,=\, \frac{S^{-T} {\bf u}(\phi)}{\|S^{-T}{\bf u}(\phi)\|}\,.
\label{n-from-u-non}
\end{equation}

In \cite{yy}, the procedure to generate the Minkowski sum boundary for solid ellipsoids
$E_1$ and $E_2$ (defined by symmetric matrices $A_1$ and $A_2$) was to morph space so as to compute
$$ \partial(A_{2}^{-1}(E_1 + E_2)) = \partial(A_{2}^{-1} \cdot  E_1 + \mathbb{B}^{N}), $$
where $ \mathbb{B}^N=\{\xx\in\IR^N:\|\xx\|<1\}$.
Since this corresponds to an external surface offset relative to $\partial(A_{2}^{-1} \cdot E_1)$ in the direction of the outward normal with a unit distance, then using {(\ref{n-from-u-non}) with $S= A_{2}^{-1} A_1$} gives the normal to the
deformed ellipsoid, and
the resulting offset surface is given parametrically by
$$ \tilde{\bf x}_{1+2}({\bf u}) = A_{2}^{-1} A_1 {\bf u} \,+\, \frac{A_{2} A_{1}^{-1} {\bf u}}{\|A_{2} A_{1}^{-1} {\bf u}\|} \,. $$
Transforming back by multiplying $\tilde{\bf x}_{1+2}({\bf u})$ by $A_2$ gives the result in \cite{yy}:
\begin{equation}
{\bf x}_{1+2}({\bf u}) =  A_1 {\bf u} \,+\, A_{2}\left(\frac{A_{2}  A_{1}^{-1} {\bf u}}{\|A_{2} A_{1}^{-1} {\bf u}\|}\right) \,.
\label{old-res}
\end{equation}
It is regrettable that
$ {\bf x}_{1+2}({\bf u}) \,\neq\, {\bf x}_{2+1}({\bf u}) $
even though $E_1 + E_2 = E_2 + E_1$.
This  led us to consider a new way of writing (\ref{old-res}). Specifically, substituting  (\ref{u-from-n}) (with
$A = A_1$) in equation \eqref{old-res} yields \begin{equation}
{\bf x}_{1+2}\left(\frac{A_1 {\bf n}}{\|A_1 {\bf n}\|}\right) =  A_1 \left(\frac{A_1 {\bf n}}{\|A_1 {\bf n}\|}\right) \,+\, A_{2}\left(\frac{A_{2} {\bf n}}{\|A_{2} {\bf n}\|}\right) \,.
\label{new-res0}
\end{equation}
Re-parametrizing \eqref{new-res0}  gives a new parametric formula for $\d(E_1+E_2)$:
\begin{equation}
{\bf x}_{\d(E_1+E_2)}(\nn) =  A_1 \left(\frac{A_1 {\bf n}}{\|A_1 {\bf n}\|}\right) \,+\, A_{2}\left(\frac{A_{2} {\bf n}}{\|A_{2} {\bf n}\|}\right) \,.
\label{new-res}
\end{equation}

The first result of this paper is that $\nn$ is normal to $\xx_{\d(E_1+E_2)}(\nn)$ and that this symmetric formula for the Minkowski sum boundary generalizes to $m$ ellipsoids:

\begin{theorem}\label{symmetric} Suppose that $E_1,E_2,\dots,E_m$ are solid ellipsoids in $\IR^N$ given by
$$E_j=\{A_j\vv:\|\vv\|<1\},\quad j=1,\dots,m,$$ where the $A_j$ are symmetric positive-definite $N\times N$ matrices.  Then the boundary of the Minkowski sum $\Sigma := E_1+\cdots+ E_m$ is given parametrically by
\begin{equation}
{\, {\bf x}_{\d\Sigma}({\bf n}) \,=\, \sum_{i=1}^{m} \frac{A_{i}^{2} {\bf n}}{\|A_i {\bf n}\|} \,,\ \ \nn\in \mathbb{S}^{N-1}.}
\label{new-mult}\end{equation}
Furthermore, $\nn$ is  the (outward-pointing) normal to the boundary of $\Sigma$ at ${\bf x}_{\d\Sigma}({\bf n})$.
\end{theorem}
We obtain Theorem~\ref{symmetric} first as a special case of a general parametric formula for the Minkowski sum of convex bodies (Theorem~\ref{sum}) and then give an alternative computational derivation in Section~\ref{proof}.

We also  have a formula for the principal curvatures of the boundary of the Minkowski sum of ellipsoids:
\begin{theorem}\label{curvatures}  Let $\Sigma=E_1+\cdots+ E_m$ be the Minkowski sum of solid ellipsoids  $E_1,\dots,E_m$ in $\IR^N$ given by symmetric positive-definite matrices   $A_1,\dots,A_m$  as in Theorem~\ref{symmetric}, and let
\beq\label{Cmatrix} C(\Sigma,\nn) := \sum_{j=1}^m \left[ \frac{A_{i}^{2}}{\|A_{i} {\bf n}\|} \,-\, \frac{A_{i}^{2} {\bf n} {\bf n}^T A_{i}^{2}}{\|A_{i} {\bf n}\|^3}\right]\,,\ \ \nn\in \mathbb{S}^{N-1}.\eeq Then $C(\Sigma,\nn)$ is positive semidefinite of rank $N-1$, and the principal curvatures of the Minkowski sum boundary $\d\Sigma$ at ${\bf x}_{\d\Sigma}({\bf n})$ are the reciprocals of the  positive eigenvalues of $C(\Sigma,\nn)$.\end{theorem}

Indeed, we show that $C(\Sigma,\nn)\nn=\bz$ and that the self-adjoint operator $C(\Sigma,\nn)$ is positive-definite on the hyperplane orthogonal to $\nn$. In fact, $C(\Sigma,\nn)$ is the Hessian matrix of the support function of the Minkowski sum $E_1+\cdots +E_m$.

We use the theory of general convex bodies in our proofs of Theorems \ref{symmetric} and \ref{curvatures}. In our derivation of Theorem~\ref{symmetric}, we give a formula for the parametrization of the boundary of the Minkowski sum of convex bodies (Theorem \ref{sum}). We also provide a  reverse isometric inequality for convex bodies (Theorem~\ref{reverse}).

\section{Parametric equation for the Minkowski sum of strictly convex bodies}

We begin by reviewing some basics about convex bodies.  An important concept for the study of Minkowski sums of convex sets is the support function \cite{Gruber,Schneider}:
\begin{definition} The support function $h_K$ of a convex body $K$ is given by
$$h_K(\vv)=\sup_{\xx\in K} \xx\cdot \vv,\quad \vv\in\IR^N.$$ \end{definition} Some elementary properties of the support function are:
\begin{enumerate}
\item[i)] $h_{K_1+K_2}=h_{K_1}+h_{K_2}$,
\item[ii)] $h_K(\vv+\ww) \le h_K(\vv)+h_K(\ww)$,
\item[iii)] $\xx\in K \iff \xx\cdot\vv <  h_K(\vv) \ \forall \vv\in\IR^N$,
\item[iv)] $K_1 \subseteq  K_2 \iff h_{K_1}(\vv) \leq h_{K_2}(\vv) \ \forall \vv\in\IR^N$,
\item[v)] $h_K(\vv)=\sup_{\|\xx\|\in\partial K} \xx\cdot \vv$.
\end{enumerate}
As a consequence of property (iv), two convex (open) bodies are identical if and only if their support functions are equal.

As an example, for a solid ellipsoid $E=\left\{A\uu:\|\uu\|<1\right\}$, where $A\in\mbox{GL}(N,\IR)$ is positive definite symmetric, we have
\beq h_E(\vv)= \sup_{\|\uu\|<1}(A\uu)\cdot \vv = \sup_{\|\uu\|=1}\uu\cdot (A\vv) =\frac{A\vv}{\|A\vv\|}\cdot (A\vv)=\|A\vv\|\,.\label{hE}\eeq
It follows by property (i) of the support function that for the Minkowski sum of solid ellipsoids given as in Theorem~\ref{symmetric}, we have
\beq\label{hEsum} h_{E_1+\cdots+E_m}(\vv)=\|A_1\vv\|+\cdots +\|A_m\vv\|\,. \eeq

Let $K$ be a convex body.
 By Theorem~4.1 of \cite{Gruber}, for every $\xx\in\d K$, there exists $\nn\in\IS^{N-1}$ such that \beq\label{x}h_K(\nn)=\xx\cdot\nn\,.\eeq
Now suppose that $K$ is strictly convex; i.e.,  for all pairs of points $\xx,\yy\in\d K$, the open line segment $\{t\xx+(1-t)\yy:0<t<1\}$ connecting $\xx$ and $\yy$ is contained in $K$. Then for each $\nn\in\sphere$, there is a unique $\xx\in\d K$ satisfying \eqref{x}. Thus we can define the  {\em normal parametrization}  of the boundary $\d K$ to be the map $ \xx_{\d K}:\IS^{N-1}\to \d K$ given by
\beq\label{xdK} h_K(\nn)=\xx_{\d K}(\nn)\cdot\nn\,, \quad \nn\in\sphere\,.\eeq

Thus for a strictly convex body $K$, the map $ \xx_{\d K}$ is well-defined and surjective, but may not be injective (for example, if $K$ is piecewise smooth with corners).  However, if $\d K$ is $\ccal^1$, then $ \xx_{\d K}$ is injective and the parameter $\nn$ is the outward-pointing unit  normal to $\d K$ at $\xx_{\d K}(\nn)$, since the maximum of $\xx\cdot\nn$ on $\d K$ is attained at the point $\xx_{\d K}(\nn)$.
In this case, $\xx_{\d K}$ is simply the inverse of the Gauss map $${\cal G}_{\d K}:\d K\to\sphere\,.$$

\begin{example} Let   $E= \{A\vv:\|\vv\|<1\}$ be a solid ellipsoid,  where $A$ is a positive-definite symmetric $N\times N$ matrix.  As shown in Section~\ref{intro}, the normal parametrization  of $\d E$ is given by equation \eqref{x-from-n}:
\beq \xx_{\d E}(\nn)=\frac{A^2\nn}{\|A\nn\|}\,.\label{x-from-n again}\eeq \end{example}

Without assuming smoothness of $\d K$, we  have the following:

\begin{lemma}{\rm \cite[Lemma~2.2.12]{Schneider}} If $K$ is a strictly convex body, then $ \xx_{\d K}$ is continuous. \end{lemma}

\begin{proof} The lemma is a consequence of the following  fact:
Let $f:X\times Y\to\IR$ be a continuous function where $X$ and $Y$ are Hausdorff spaces and $X$ is compact.  For $y\in Y$, let $f_y(x) = f(x,y)$. Suppose that $f_y$ attains its maximum at a unique point $\psi(y)\in X$ for all $y\in Y$.  Then the map $\psi:Y\to X$ is continuous.

To prove this fact, let $y_0\in Y$ and let $U\subset X$ be an arbitrary neighborhood of $x_0 := \psi(y_0)$. Then $\sup_{X\smallsetminus U}f_{y_0}<f_{y_0}(x_0)$ by compactness of $X$ and uniqueness of the maximum point. Therefore, there is a neighborhood $V\subset Y$ of $y_0$ such that $\sup_{X\smallsetminus U}f_{y}<f_{y}(x_0)$ for all $y\in V$ and hence $\psi(y)\in U$ for all $y\in V$, verifying that $\psi$ is continuous at $y_0$.

We apply this  with $f:K\times\sphere\to \IR$ given by $f(\xx,\nn)=\xx\cdot\nn$.  Then $\psi(\nn)=\xx_{\d K}(\nn)$ is continuous.\end{proof}

\begin{theorem} \label{sum} Let $K_1,\dots,K_m$ be strictly convex bodies in $\IR^N$. Then the normal parametrization of the boundary of the Minkowski sum $K_1+\cdots+ K_m$ is given by
\beq\label{sumeq} \xx_{\d(K_1+\cdots+K_m)}(\nn)=\xx_{\d K_1}(\nn)+\cdots+\xx_{\d K_m}(\nn)\,,\quad \nn\in\sphere\,.\eeq\end{theorem}

\begin{proof} It suffices to consider the case $m=2$. For $\nn\in\sphere$, we have \begin{multline*}\xx_{\d(K_1+K_2)}(\nn)\cdot \nn=\sup_{\xx\in K_1+K_2}\xx\cdot \nn =\sup_{\xx_1\in K_1}\xx_1\cdot \nn +\sup_{\xx_2\in K_2}\xx_2\cdot \nn\\ = \xx_{\d K_1}(\nn)\cdot \nn+\xx_{\d K_2}(\nn)\cdot \nn= \left[\xx_{\d K_1}(\nn)+\xx_{\d K_2}(\nn)\right]\cdot \nn\,.\end{multline*}
The conclusion then follows from uniqueness of $\xx_{\d(K_1+K_2)}(\nn)$.\end{proof}

A special case of Theorem \ref{sum}  is the formula in \cite{ruan2020} for $\xx_{\d(K_1+(-K_2))}(\nn)$ where $K_1,K_2$ have smooth boundaries.

\begin{proof}[Proof of Theorem \ref{symmetric}] The result is an immediate consequence of Theorem \ref{sum} with $K_j=E_j$ and equation \eqref{x-from-n again}. \end{proof}

We give an alternative derivation of Theorem \ref{symmetric} in the following section.

\subsection{Direct Proof of Theorem \ref{symmetric}}\label{proof}
The result holds for $m=1$ by \eqref{x-from-n again}.
Using induction, we let $m\ge 1$ and assume that the formula in (\ref{new-mult}) holds for $\Sigma=E_1+\cdots+E_m$. As in the proof of
(\ref{old-res}), we shall apply $A_{m+1}^{-1}$ to $\Sigma+E_{m+1}$  and compute a normal offset.

The fact that ${\bf n}$ is the outward-pointing unit normal to ${\bf x}_{\d(E_1+\cdots+E_m)}({\bf n})$   follows by an elementary
computation. Specifically, if ${\bf n} = {\bf n}(\phi)$ where $\phi = (\phi_1,...,\phi_{N-1})$ are the spherical angles
parametrizing the unit sphere (or are any local coordinates on the unit sphere), then a basis for the tangent hyperplane at the point ${\bf x}_{\d(E_1+\cdots+E_m)}({\bf n})$  on this surface
will be the  $N-1$ tangent vectors $\partial {\bf x}_{\d(E_1+\cdots+E_m)}({\bf n}(\phi_1,...,\phi_{N-1}))/\partial \phi_j$. These tangent vectors will be a sum over $i$ of
\begin{equation}
\frac{\partial}{\partial \phi_j} \left(\frac{A_{i}^{2} {\bf n}}{\|A_{i} {\bf n}\|}\right) =
C(E_i,\nn) \frac{\partial {\bf n}}{\partial \phi_j}
\label{deriv-phi-j}
\end{equation}
where
\begin{equation}
C(E_i,\nn) \, :=\, \frac{A_{i}^{2}}{\|A_{i} {\bf n}\|} \,-\, \frac{A_{i}^{2} {\bf n} {\bf n}^T A_{i}^{2}}{\|A_{i} {\bf n}\|^3} \,=\, [C(E_i,\nn)]^T\,.
\label{C-def}
\end{equation}
We note that for $\vv\in\IR^N$, we have \beq \vv^TC(E_i,\nn)\vv\, =\,\frac{\|A_i\nn\|^2\|A_i\vv\|^2 - [(A_i\vv)^T (A_i\nn)]^2}{ \|A_1\nn\|^3}\,\ge\,0\,,\label{Cgeq0}\eeq with equality if and only if $\vv=c\nn$.
It follows that\begin{equation}
[C(E_i,\nn)] {\bf n} \,=\, {\bf 0}\,.
\label{Cizero}\end{equation}
Thus
\begin{equation*}
{\bf n}^T \frac{\partial}{\partial \phi_j} \left(\frac{A_{i}^{2} {\bf n}}{\|A_{i} {\bf n}\|}\right) \,=\,\nn^TC(E_i,\nn)\frac{\partial {\bf n}}{\partial \phi_j}\,=\, {\bf 0}\,,
\label{}
\end{equation*}
and consequently
\begin{equation}
{\bf n}^T \frac{\partial {\bf x}_{\d(E_1+\cdots+E_m)}({\bf n}(\phi_1,...,\phi_{N-1}))}{\partial \phi_j} = 0\,.
\label{nnndndn}
\end{equation}
Hence, ${\bf n}$ is the normal to ${\bf x}_{\d(E_1+\cdots+E_m)}({\bf n})$.

It follows that $A_{m+1} \nn$ is normal to  $A_{m+1}^{-1} {\bf x}_{\d(E_1+\cdots+E_m)}({\bf n})$.
Therefore, the Minkowski sum of $A_{m+1}^{-1} {\bf x}_{\d(E_1+\cdots+E_m)}({\bf n})$ with the unit ball $\mathbb{B}^N= A_{m+1}^{-1}\cdot E_{m+1}$ has boundary given parametrically as
$$ {\bf x}_{\d A_{m+1}^{-1}(E_1+\cdots+E_{m+1})}({\bf n}) \,=\
A_{m+1}^{-1} {\bf x}_{\d(E_1+\cdots+E_m)}({\bf n}) + \frac{A_{m+1} {\bf n}}{\|A_{m+1} {\bf n}\|}\,. $$
The boundary of the Minkowski sum of $m+1$ ellipsoids is then given parametrically by
$$ {\bf x}_{\d(E_1+  \cdots +  E_{m+1})}({\bf n}) \,=\ A_{m+1}
{\bf x}_{\d A_{m+1}^{-1}(E_1+  \cdots + \ E_{m+1})}({\bf n})  \,=\, \sum_{i=1}^{m+1} \frac{A_{i}^{2} {\bf n}}{\|A_i {\bf n}\|} \,,$$ which is exactly  the formula in (\ref{new-mult}) with $m \mapsto m+1$.\qed

\section{Isoperimetric inequalities}\label{section-reverse}
The classical isoperimetric inequality gives a lower bound for the volume of the boundary of a bounded body $X\subset \IR^N$:
\beq \vol_{N-1}(\d X)\ge N \,\vol_{N}(X)^{\frac{N-1}{N}}\vol_N(\IB^N)^{\frac{1}{N}}\,.\label{iso}\eeq
Here $\vol_N$ denotes Lebesgue measure in $\IR^N$, and \beq\label{voldef}\vol_{N-1}(S) :=\liminf_{\epsilon\to 0}  \frac{\vol_N\{\xx\in\IR^N:\dist(\xx,S)<\epsilon\}}{2\epsilon}\;.\eeq (This general isometric inequality is given in Federer \cite[3.2.43]{Federer}.) If $\d X$ is piecewise smooth, then $\vol_{N-1}(\d X)$ is the usual $(N-1)$-dimensional volume.  In this case, \eqref{iso} can be reduced to the case where $X$ is convex, where the result is a consequence of the Brunn-Minkowski inequality (see \cite{gardner, Schneider}). More generally, if  $\d X$ is ($N-1$)-rectifiable, then $\vol_{N-1}(\d X)= {\mathcal H}^{N-1}(X)$, the Hausdorff ($N-1$)-measure \cite[3.2.39]{Federer}. Now let $r(X)$ and $R(X)$ denote the inradius and circumradius, respectively, of $X$.  Noting that
$$\left[\frac{\vol(\IB^N)}{\vol(X)}\right]^{1/N} \ge \left[\frac{\vol(\IB^N)}{\vol\left(R(X)\IB^N\right)}\right]^{1/N}=\frac 1{R(X)}\,,$$
the isoperimetric inequality \eqref{iso} immediately implies that
\beq\label{iso1}  \vol_{N-1}(\d X)\ge \frac N{R(X)} \,\vol_{N}(X)\,.\eeq

A sharper form of \eqref{iso1} is given by the following:

\begin{theorem}\label{iso1+t} Let $E =  \{A\vv+\yy:\vv\in\IB^N\}$ be the minimal volume outer ellipsoid (MVOE) of a bounded body $X$ in $\IR^N$. Then
\beq  \label{iso1+}  \vol_{N-1}(\d X)\ge \frac N{(\det A)^{1/N}} \,\vol_{N}(X)\,.\eeq \end{theorem}
\begin{proof} The bound \eqref{iso1+} follows immediately from \eqref{iso} and
$$\left[\frac{\vol(\IB^N)}{\vol(X)}\right]^{1/N} \ge \left[\frac {\vol(\IB^N)} {\vol(E)} \right]^{1/N}
= (\det A)^{-1/N}\,.$$
\end{proof}
Theorem \ref{iso1+t} is used in Section \ref{upper} to give an upper bound on the volume of the Minkowski sum of ellipsoids (Theorem \ref{upperboundt}).

Some reverse isometric inequalities for convex bodies based on affine transformations are given in \cite{Ball91} and \cite{grinberg}. We give here a  reverse isometric inequality analogous to \eqref{iso1} and \eqref{iso1+}.

\begin{theorem}\label{reverse}  Let $K$ be a convex body. Then  $\vol_{N-1}(\d K)= {\mathcal H}^{N-1}(K)$ and
\beq\label{iso2}  \vol_{N-1}(\d K)\le \frac N{r(K)} \,\vol_{N}(K)\,.\eeq
If $K$ is strictly convex or if $\d K$ is $\ccal^1$, then \eqref{iso2} is an equality if and only if $K$ is a ball.\end{theorem}
\begin{proof} By  a translation, we  assume throughout this proof that $r(K)\IB^N\subseteq K$.

We first consider the case where $K$ has a smooth boundary.  By the divergence theorem,
\begin{multline} \vol_N(K) = \int_K1\,d\vol_N = \frac 1N \int_K \nabla \cdot \xx\, d\vol_N(\xx)\\ = \frac 1N\int_{\d K} \xx\cdot  {\mathcal G}_{\d K}(\xx)\,d\vol_{N-1}(\xx)\,.\label{divergence}\end{multline}
Writing $\nn={\mathcal G}_{\d K}(\xx)$, for $\xx\in\d K$, we have
\beq\label{xu} \xx\cdot  {\mathcal G}_{\d K}(\xx)=\xx \cdot \nn= h_K(\nn)\ge h_{r(K)\IB^N}(\nn)=r(K)\,.\eeq
Therefore by \eqref{divergence},
\beq\vol_N(K) \ge \frac 1N\int_{\d K} r(K)\,d\vol_{N-1}(\xx) = \frac {r(K)}{N}\vol_{N-1}(\xx)\,,\label{=iso2}\eeq verifying \eqref{iso2} for the case where $\d K$ is smooth.

Now suppose $K$ is not a ball and $\d K$ is $\ccal^1$. We must show strict inequality in \eqref{iso2}. Since $K\supsetneqq r(K)\IB^N$, there exists  $\xx_0\in\d K$ such that $$\|\xx_0\|=\sup_{\xx\in\d K}\|\xx\|>r(K)\,.$$ Therefore $\nabla (\|\xx\|^2) =2\xx$ is orthogonal to $\d K$ at $\xx_0$, and hence ${\mathcal G}_{\d K}(\xx_0)=\|\xx_0\|^{-1}\xx_0$. Thus, $\xx_0\cdot{\mathcal G}_{\d K}(\xx_0)=\|x_0\|>r(K)$ and hence strict equality holds in \eqref{=iso2}. This completes the argument that for smooth convex bodies $K$, equality in \eqref{iso2} implies $K$ is a ball. The converse is trivial.

Now let $K$  be a general convex body. By \cite[3.2.35]{Federer}, $\d K$ is ($N-1$)-rectifiable and hence $\vol_{N-1}(\d K)= {\mathcal H}^{N-1}(\d K)$ \cite[3.2.39]{Federer}.  By \cite[Theorem~2.2.6]{Schneider}, $\bar K$ is the intersection of half spaces, and therefore is the intersection of a decreasing sequence of convex polytopes $\{K_i\}$. Since $\vol_{N-1}(K_i)\le \vol_{N-1}(K_1)<\infty$, it follows that the gradient (in the distribution sense) of the characteristic function of $K$  is a measure (i.e., $K$ has {\it finite perimeter\/}; see \cite[p.~9]{Fleming}), and hence by a theorem of  De Georgi \cite{DeG} (see also \cite[4.5.6]{Federer}),  the divergence theorem \eqref{divergence} holds for $K$ with ${\mathcal G}_{\d K}(\xx)$ given by \beq\|{\mathcal G}_{\d K}(\xx)\|=1 \quad\mbox{and}\quad \xx\cdot {\mathcal G}_{\d K}(\xx)=h_K({\mathcal G}_{\d K}(\xx))\,, \label{aenormal}\eeq  for ${\mathcal H}^{N-1}$-a.a. $\xx\in \d K$.
(By \cite[4.5.6]{Federer} or \cite[Theorem~2.2.5]{Schneider}, there exists a unique ${\mathcal G}_{\d K}(\xx)$ satisfying \eqref{aenormal} for ${\mathcal H}^{N-1}$-a.a. $\xx\in\d K$.) Then \eqref{iso2} follows from the divergence theorem as before.

Finally, we consider the case where $K$ is strictly convex and not necessarily smooth. Suppose that equation \eqref{iso2} is an equality. Then \eqref{xu} is an equality for  ${\mathcal H}^{N-1}$-a.a. $\xx\in\d K$. I.e., there exists a  set $Z\subset \d K$ such that ${\mathcal H}^{N-1}(\d K\smallsetminus Z)=0$ and
\beq\xx\cdot  {\mathcal G}_{\d K}(\xx)=  r(K)\quad\forall\;\xx\in Z\,.\label{byxu}\eeq

Now fix a point $\xx_0\in Z$.\ and let $P$ be the support hyperplane for $K$ through $\xx_0$; i.e., $\xx_0\in P$ and ${\mathcal G}_{\d K}(\xx_0)$ is orthogonal to $P$. Then by \eqref{byxu}, $K$ is contained in the half-space  $\{\vv\in\IR^N: \vv\cdot{\mathcal G}_{\d K}(\xx_0)<r(K)\}$, which is bounded by $P$; in particular $K\cap P=\emptyset$.
Let $$\yy=r(K){\mathcal G}_{\d K}(\xx_0) \in P\cap  r(K)\sphere\subseteq P\cap \bar K\subset \d K\,.$$ We claim that $\yy=\xx_0$.  Indeed, suppose on the contrary that $\yy\neq\xx_0$. Then by the  strict convexity of $\bar K$, the midpoint $\zz= \frac12 \yy+\frac12\xx_0\in P\cap K$, contradicting the fact that $K\cap P=\emptyset$. Therefore $\xx_0=\yy$ and thus $\|\xx_0\|=\|\yy\|=r(K)$; i.e. $\xx_0\in r(K)\sphere$.

Since $\xx_0\in Z$ is arbitrary,  we have shown that  $Z\subseteq r(K)\sphere$ and therefore $\bar Z\subseteq r(K)\sphere$. Since ${\mathcal H}^{N-1}(\d K\smallsetminus \bar Z)=0$, and all nonempty, relatively open subsets of $\d K$ have positive ${\mathcal H}^{N-1}$-measure zero, it follows that $\bar Z=\d K\subseteq  r(K)\sphere$ and thus $K\subseteq r(K){ \IB^N}$. Therefore $K= r(K){ \IB^N}$.\end{proof}

The inequality \eqref{iso2} for  planar convex bodies ($N=2$) is stated in \cite[p.~90]{santalo}.
Equality in \eqref{iso2} can occur  for non-strictly-convex bodies, for example, when $K$ is an arbitrary triangle in $\IR^2$, or when $K$ is an $N$-cube in $\IR^N$. In general, if equality holds in \eqref{iso2} for a convex body $K\supseteq r(K)\IB^N$, then for every point $\xx\in\d K\smallsetminus  r(K)\sphere$, there exists a point $\yy\in  r(K)\sphere$ such that the straight line segment from $\xx$ to $\yy$ is contained in $\d K$ (and is tangent to $ r(K)\sphere$ at $\yy$).  To give another  example, let $Q$ be a cone in $\IR^3$ tangent to $\IS^2$ along a horizontal circle $C\subset \IS^2$ (with the vertex of $Q$ below the sphere); then equality holds in \eqref{iso2}  for the `ice cream cone' shaped body $K$ bounded by the portion of $Q$ below $C$ and the portion of  $\IS^2$ above $C$.

An alternative proof of \eqref{iso2} without geometric measure theory is as follows: As mentioned above, we can choose a decreasing sequence of convex polytopes $K_i$, such that $\bigcap K_i=\bar K$. We  note that the divergence theorem \eqref{divergence} holds for convex polytopes (for an elementary proof, see \cite[p.~10]{Stolzenberg} or \cite[Ch.~III]{W}), and \eqref{xu} holds for all smooth points $\xx\in\d K$. It then follows as before that  \eqref{iso2} holds for the polytopes $K_i$. Since $K_i\searrow K$ as $i\to\infty$, we have $\vol_N(K_i)\to \vol_N(K)$. Furthermore, for convex bodies $D\subset\IR^N$,
\beq\vol_{N-1}(\d D) = \frac {N\,\vol(\IB^N)}{\vol(\IB^{N-1})} \int_{\sphere}\vol_{N-1}[\Pi_{\uu^\perp}(D)]\,d\uu\,,\label{quer1}\eeq where $\Pi_{\uu^\perp}$ is the projection onto the hyperplane  in $\IR^N$ orthogonal to $\uu$, and $d\uu$ is the invariant probability measure on $\sphere$ (see \cite{gardner,Schneider}). It  follows from \eqref{quer1} and the Lebesgue dominated convergence theorem that $\vol_{N-1}(\d K_i)\to \vol_{N-1}(\d K)$. Since $r(K)\le r(K_i)$, we have
$$\vol_{N-1}(\d K_i)\le \frac N{r(K_i)} \,\vol_{N}(K_i) \le \frac N{r(K)} \,\vol_{N}(K_i)\,.$$
Therefore,
$$\vol_{N-1}(\d K) = \lim_{i\to\infty}\vol_{N-1}(\d K_i)\le \frac N{r(K)} \lim_{i\to\infty} \,\vol_{N}(K_i)=\frac N{r(K)} \,\vol_{N}(K).$$ \qed

\section{Differential Geometry of Strictly Convex Bodies: Proof of Theorem~\ref{curvatures}}

In the following, we let $K$ be a  convex body such that its boundary $\d K$ is $\ccal^2$ and has strictly positive principal curvatures. This condition implies that $K$ is strictly convex (but strict convexity does not imply positive principal curvatures).
We define $\hat\xx_{\d K}:\IR^N\to \d K$  by
\beq\label{xdK+} \hat\xx_{\d K}(\vv) := \xx_{\d K}\left(\frac{\vv}{\|\vv\|}\right) = {\cal G}_{\d K}^{-1} \left(\frac\vv{\|\vv\|}\right),\eeq where we recall that ${\cal G}_{\d K}$ denotes the Gauss map of $\d K$. We note that $\xx_{\d K}=\hat\xx_{\d K}|_{\sphere}$.

\begin{lemma} If $\d K$ is $\ccal^2$ and has positive principal curvatures, then $\hat\xx_{\d K}$ is $\ccal^1$ on $\IR^N\smallsetminus \{\bz\}$.\label{inverseGauss}\end{lemma}

\begin{proof} If $\d K$  is $\ccal^2$, then the Gauss map ${\cal G}_{\d K}$ is $\ccal^1$.  Thus, it suffices to show that the Jacobian of  ${\cal G}_{\d K}$ is invertible at each  point of $\d K$ and hence its inverse $\xx_{\d K}:\sphere\to\d K$ is $\ccal^1$. Let $\xx_0\in\d K$, and choose local coordinates $t_1,t_2,\dots,t_{N-1}$  on $\d K$ in a neighborhood of $\xx_0$ .  Let $\phi_1,\dots,\phi_{N-1}$ be local coordinates on $\sphere$ in a neighborhood of ${\cal G}_{\d K}(\xx_0)$. We then can write
$${\cal G}_{\d K}(\xx(t))=\nn(\phi_1(t),\dots,\phi_{N-1}(t))\,,\quad t=(t_1,\dots,t_{N-1})\,,$$ where $\xx(t)$ denotes the point in $\d K$ with coordinates  $t_1,t_2,\dots,t_{N-1}$, and similarly $\nn(\phi_1,\dots,\phi_{N-1})$ denotes the point in $\sphere$ with coordinates $\phi_1,\dots,\phi_{N-1}$ .

The Second Fundamental Form $[l^t_{ij}]$ ($1\le i,j\le N-1$) for $\d K$ with respect to the coordinates $(t_1,\dots,t_{N-1})$ is given by
\beq\label{Ltij} l^t_{ij}=  -\nn\cdot\frac{\partial^2 {\bf x}}{\partial t_i \partial t_j}= \frac{\d \nn}{\d t_i}\cdot\frac{\partial {\bf x}}{ \partial t_j} =\sum_{k=1}^{N-1}\frac{\d\phi_k}{\d t_i}\,\frac{\d \nn}{\d \phi_k}\cdot\frac{\partial \xx}{ \partial t_j}=\sum_{k=1}^{N-1} J_{ik}\,W_{kj}\,,\eeq where  $(J_{ki})= (\frac{\d\phi_k}{\d t_i})$ is the Jacobian of ${\cal G}_{\d K}$ with respect to the coordinates $t_i$ and $\phi_k$, and $W_{kj}=\frac{\d \nn}{\d \phi_k}\cdot\frac{\partial \xx}{ \partial t_j}$.  If the principal curvatures of $\d K$ are non-zero, then $(l^t_{ij})$ is non-singular, and hence  $(J_{ik})$ is non-singular.\end{proof}

 We now use the notation   $D_jf(v_1,\dots,v_N)=\d f/\d v_j$.
\begin{definition} For $K$ as in Lemma \ref{inverseGauss} and $\nn\in\sphere$, we  define the $N\times N$ {\em convexity matrix}   $C(K,\nn)$ by\beq\label{Cdef} C(K,\nn)_{ij}=D_j x_i(\nn)\,,\quad 1\le i,j\le N\eeq
where we write $\hat\xx_{\d K}(\nn)=[x_1(\nn),\dots,x_n(\nn)]^T$. I.e., $C(K,\nn)$ is the Jacobian of $\hat\xx_{\d K}$ at $\nn$.
\end{definition}

The following result is given in Schneider \cite[p.~115]{Schneider}:
\begin{lemma} For $K$ as in Lemma \ref{inverseGauss}, the support function $h_K$ is $\ccal^2$ and \beq\hat\xx_{\d K}(\vv) = \nabla h_K(\vv)\,,\quad\forall\ \vv\in\IR^N\smallsetminus\{\bz\}\,,\label{gradient}\eeq and therefore \beq\label{Hessian} C(K,\nn)_{ij}=D_iD_j h_K(\nn)\,.\eeq \end{lemma}
\begin{proof}   To verify \eqref{gradient}, we note that for $\vv\in\IR^N\smallsetminus\{\bz\}$,
$$h_K(\vv)=\sup _{\yy\in K}(\vv\cdot\yy)=\vv\cdot\xx_{\d K}\left(\frac \vv{\|\vv\|}\right) = \vv\cdot\hat\xx_{\d K}(\vv)\,.$$ Since $\hat\xx_{\d K}$ is $\ccal^1$ by Lemma~\ref{inverseGauss}, $h_K$ is also  $\ccal^1$. Using the above notation, we have $$D_jh_K(\vv)= x_j(\vv)+\vv\cdot D_j\hat\xx_{\d K}(\vv)\,.$$
Since the tangent hyperplane to $\d K$ at $\hat\xx_{\d K}(\vv)$ is orthogonal to $\vv$, we conclude that  $D_jh_K(\vv)= x_j(\vv)$, which yields \eqref{gradient}.  Since $D_jh_K=x_j$ is $\ccal^1$, it follows that $h_K$ is $\ccal^2$. Equation \eqref{Hessian} is an immediate consequence of \eqref{Cdef}--\eqref{gradient}.\end{proof}

To prove Theorem \ref{curvatures}, we shall use the following result of Blaschke (in dimension 3) and Firey \cite{Firey} (see also   \cite[Cor.~2.5.2]{Schneider}):

\begin{theorem}\label{Kcurvatures} Let $K\subset\IR^N$ be a convex body such that  its boundary $\d K$ is $\ccal^2$ and has positive sectional curvatures. Then for all $\nn\in\sphere$,
\begin{enumerate} \item[i)] $C(K,\nn)$ is symmetric, \item[ii)]$C(K,\nn)\nn=\bz$, \item[iii)] $C(K,\nn)$ is positive-definite on $\nn^\perp$, \item[iv)]  the principal curvatures of the Minkowski sum boundary $\d K$ at ${\bf x}_{\d K}({\bf n})$ are the reciprocals of the  non-zero eigenvalues of $C(K,\nn)$.\end{enumerate}\end{theorem}

\begin{proof} Conclusion (i) follows from \eqref{Hessian}. To verify (ii), we note that \eqref{Cdef} implies that $$C(K,\nn)\vv= \sum_{j=1}^Nv_jD_j\hat\xx_{\d K}(\nn)= \frac d{ds}\hat\xx_{\d K}(\nn+s\vv)|_{
s=0}\,.$$
Since $\hat\xx_{\d K}(\nn+s\nn)=\hat\xx_{\d K}(\nn)$, it follows that $C(K,\nn)\nn=\bz$.

Let $\phi_1,\dots,\phi_{N-1}$ denote local coordinates of points $\nn(\phi_1,\dots,\phi_{N-1})\in\sphere$. Define $$\xx(\phi_1,\dots,\phi_{N-1}) :=\xx_{\d K}(\nn(\phi_1,\dots,\phi_{N-1}))\in\d K\,,$$ so that $\phi_1,\dots,\phi_{N-1}$ are also local coordinates on $\d K$.  We then have
\beq\label{Cformula} \frac{\d x_i}{\d \phi_k} = \sum_{j=1}^N\frac {\d x_i}{\d n_j }\frac{\d n_j}{\d \phi_k}  = \sum_{j=1}^N C(K,\nn)_{ij}\frac{\d n_j}{\d \phi_k}\,.\eeq

The elements of the metric tensor  and the second fundamental form for $\d K$ are given, respectively, by
\begin{eqnarray}
g_{ij} &=&\frac{\partial {\bf x}}{\partial \phi_i} \cdot \frac{\partial {\bf x}}{\partial \phi_j}\,,\notag\\
l_{ij} &=& -\nn\cdot\frac{\partial^2 {\bf x}}{\partial \phi_i \partial \phi_j} \ =\ \frac{\partial {\bf n}}{\partial \phi_i} \cdot \frac{\partial {\bf x}}{\partial \phi_j} \,.
\label{2ndfundform}
\end{eqnarray} (The sign of $l_{ij}$ is chosen so that the principal curvatures of $\d K$ are positive.)
Then with $G=[g_{ij}]$ and $L=[l_{ij}]$, the $N-1$ principal curvatures are obtained as the eigenvalues of the matrix $G^{-1} L$.

We now write $C=C(K,\nn)$. By \eqref{Cformula},
\begin{equation}
\frac{\partial {\bf x}}{\partial \phi_i} \,=\, C \frac{\partial {\bf n}}{\partial \phi_i} \,,
\label{recall}
\end{equation} and so
\beq
g_{ij} = \frac{\partial {\bf n}^T}{\partial \phi_i} C^TC \frac{\partial {\bf n}}{\partial \phi_j} = {\frac{\partial {\bf n}^T}{\partial \phi_i}   C^2 \frac{\partial {\bf n}}{\partial \phi_j}}\,,
\label{gijmink}
\end{equation} where we have used the symmetry of $C$.
The  metric tensor $G=(g_{ij})$ of $\d K$ can then be written as
\begin{equation}
G = J_{\mathbb{S}^{N-1}}^T   C^2 J_{\mathbb{S}^{N-1}}
\label{Gmink}
\end{equation}
where
\begin{equation}
J_{\mathbb{S}^{N-1}} \,=\,\left[\frac{\partial {\bf n}}{\partial \phi_1}\ \frac{\partial {\bf n}}{\partial \phi_2}\ ...\ \frac{\partial {\bf n}}{\partial \phi_{N-1}}\right]
\label{dd-1jac}
\end{equation}
is the Jacobian for the sphere. The metric tensor of the sphere is
\begin{equation}
G_{\mathbb{S}^{N-1}} \,=\, J_{\mathbb{S}^{N-1}}^T J_{\mathbb{S}^{N-1}} \,.
\label{Gsphere}
\end{equation}
Note that since $\nn\cdot ({\partial {\bf n}}/{\partial \phi_i})=0$, we have
\begin{equation} J^T_{\mathbb{S}^{N-1}}{\bf n} \,=\, {\bf 0}\,.\label{northjab}\end{equation}
In the above equations and in the following, all matrices are functions of ${\bf n}$.

Let $\{\vv^1(\nn),\dots,\vv^{N-1}(\nn)\}$ be a (moving) orthonormal basis  for the tangent hyperplane to $\d K$  at $\xx_{\d K}(\nn)$; this hyperplane is orthogonal to the null direction $\nn$ of $C$.
Since $C$ is self-adjoint (symmetric), the range of $C$ is also spanned by the $\vv^i(\nn)$. Thus we can express $C$  in terms of these basis vectors.  Explicitly,  we define the  $N\times (N-1)$ matrix
\begin{equation}
M \, :=\,\left[\vv^1(\nn)\; \vv^2(\nn)\;\dots\; \vv^{N-1}(\nn)\right],
\label{mdef-nndkfvv}
\end{equation}
which has the properties
\beq\label{MM} M^T M \,=\, \mathbb{I}_{N-1}\,,\,\,\, MM^T+\nn\nn^T=\II_N\,, \,\,\,{\rm and}\,\,\,
{\bf n}^T M \,=\,0\,. \eeq
Then
\beq\label{MCM} \tilde  C(K,\nn) := M^T \,C(K,\nn)\, M  \eeq is the matrix of the operator $C(K,\nn)$ on the tangent hyperplane at $\xx_{\d K}(\nn)$ with respect to the orthonormal basis $\{\vv^1(\nn),\dots,\vv^{N-1}(\nn)\}$. We similarly write $\tilde C= \tilde  C(K,\nn)$.

Moreover, we can write
\beq \tilde{J}_{\mathbb{S}^{N-1}} \,=\,
M^T J_{\mathbb{S}^{N-1}} \,. \label{MTJ}\eeq It then follows from \eqref{northjab} and \eqref{MM} that
\beq \label{MJtilde} {J}_{\mathbb{S}^{N-1}}=(MM^T+\nn\nn^T) {J}_{\mathbb{S}^{N-1}}  =M\tilde {J}_{\mathbb{S}^{N-1}}\,,\eeq
and therefore by \eqref{Gsphere} and \eqref{MM}, \beq\label{Gsphere1}
G_{\mathbb{S}^{N-1}} \,=\, \tilde{J}_{\mathbb{S}^{N-1}}^T \tilde{J}_{\mathbb{S}^{N-1}}\,.\eeq

By (ii) and \eqref{MM}--\eqref{MCM}, $$\tilde C^2=M^TC(MM^T+\nn\nn^T)CM= M^TC^2M\,,$$ and hence
\begin{equation}
G = \tilde{J}_{\mathbb{S}^{N-1}}^T {\tilde C^2}\tilde{J}_{\mathbb{S}^{N-1}}\,.
\label{Gsquare}
\end{equation}

By  \eqref{2ndfundform}--\eqref{recall}, we have \begin{equation}
l_{ij} \,=\, \frac{\partial {\bf n}^T}{\partial \phi_i} C \frac{\partial {\bf n}}{\partial \phi_j} \,,
\label{lij}
\end{equation}
or
\begin{equation}
L = {J}_{\mathbb{S}^{N-1}}^T C {J}_{\mathbb{S}^{N-1}} \,.
\label{Lksnksdskflf}
\end{equation}
As with $G$, this can be expressed in terms of square matrices { by using \eqref{northjab} and \eqref{MM} to obtain}
\begin{equation}
L =  \tilde{J}_{\mathbb{S}^{N-1}}^T \tilde{C} \tilde{J}_{\mathbb{S}^{N-1}} \,.
\label{Laskhvcchfa}
\end{equation} Hence $\tilde C$ is invertible and
\beq G^{-1} L = \left(\tilde{J}_{\mathbb{S}^{N-1}}^{-1}\tilde{C}^{-2}
\tilde{J}_{\mathbb{S}^{N-1}}^{-T}\right)\left(
\tilde{J}_{\mathbb{S}^{N-1}}^T \tilde{C} \tilde{J}_{\mathbb{S}^{N-1}}\right)= \tilde{J}_{\mathbb{S}^{N-1}}^{-1} \tilde{C}^{-1} \tilde{J}_{\mathbb{S}^{N-1}}\,. \eeq
Using the fact that eigenvalues are invariant under similarity transformations, the principal curvatures $\kappa_i=\kappa_i(\xx)$ of $\d K$ can then be computed as
\begin{equation}
{\kappa_i(\xx_{\d K}(\nn)) \,=\, \lambda_i\big(\tilde C(K,\nn)^{-1}\big) \,=\, \frac{1}{\lambda_i(C(K,\nn))} \,,  \quad i=1,\dots,N-1\,,}
\label{kappa2kdsnwdcw}
\end{equation}  where $\lambda_i$ denotes the $i$-th eigenvalue.
The last equality  follows from the fact that the eigenvalues of the $(N-1)\times(N-1)$ matrix $\tilde{C}$ are the same as those for the $N\times N$ matrix $C$, except for the single zero eigenvalue  $\lambda_N(C)=0$ corresponding to the eigenvector $\nn$ by (ii). This verifies (iii)--(iv) and completes the proof of Theorem~\ref{Kcurvatures}.\end{proof}

\begin{remark}
If the $\phi_i$  are the standard spherical angles on $\IS^{N-1}$, then the tangent vectors $\partial\nn/\partial\phi_i$ at $\xx_{\d K}(\nn)$ are orthogonal and   we could choose the orthonormal basis used in equation \eqref{mdef-nndkfvv} to be
\beq\label{vspherical}\vv^i(\nn)=\left\|\frac{\partial {\bf n}}{\partial \phi_i}\right\|^{-1} \frac{\partial {\bf n}}{\partial \phi_i}\;, \quad 1\le j\le N-1.\eeq
With the choice \eqref{vspherical}, $\tilde{J}_{\mathbb{S}^{N-1}}$ and $G_{\mathbb{S}^{N-1}}$ are diagonal matrices: \begin{eqnarray*}\tilde{J}_{\mathbb{S}^{N-1}}&=&\diag\left(\left\|\frac{\partial {\bf n}}{\partial \phi_1}\right\|\,,\dots,\left\|\frac{\partial {\bf n}}{\partial \phi_{N-1}}\right\|\right)\,,\\G_{\mathbb{S}^{N-1}}&=&\diag\left(\left\|\frac{\partial {\bf n}}{\partial \phi_1}\right\|^2\,,\dots,\left\|\frac{\partial {\bf n}}{\partial \phi_{N-1}}\right\|^2\right).\end{eqnarray*}
\end{remark}

\begin{proof}[Proof of Theorem \ref{curvatures}] It follows from Theorem \ref{sum} that
\beq\label{Csum} C({K_1+\cdots+K_m},\nn)= C({K_1},\nn)+\cdots +C({K_m},\nn)\,,\quad \nn\in\sphere\,.\eeq Thus by \eqref{x-from-n again} and \eqref{deriv-phi-j}-\eqref{C-def}, the convexity matrix $C(\Sigma,\nn)$ of $\Sigma$ is given by
\begin{equation}
C(\Sigma,\nn) \,=\, \sum_{i=1}^{m} C({E_i},\nn) \,=\,  \sum_{j=1}^m \left[ \frac{A_{i}^{2}}{\|A_{i} {\bf n}\|} \,-\, \frac{A_{i}^{2} {\bf n} {\bf n}^T A_{i}^{2}}{\|A_{i} {\bf n}\|^3}\right]\,,
\label{C1-def}
\end{equation}
which agrees with \eqref{Cmatrix}. Theorem~\ref{Kcurvatures}  provides the stated properties of $C(\Sigma,\nn)$.
\end{proof}

The $\tilde C$ matrix is  useful for computing integrals over the boundary of ellipsoidal sums.
Invariant integration on the sphere is given  by
\begin{equation}
\int_{\mathbb{S}^{N-1}} f({\bf n}) d\sigma_{N-1}({\bf n}) = \int
f({\bf x}(\phi))\,\left[\det G_{\mathbb{S}^{N-1}}(\phi)\right]^{1/2}\,d\phi_1\cdots d\phi_{N-1}\,,
\label{sphere int}
\end{equation}  where $\sigma_{N-1}$ is volume measure on $\mathbb{S}^{N-1}$.
By \eqref{Gsphere1}--\eqref{Gsquare}, $\det G =(\det\tilde C)^2\det G_{\sphere}$. Integration over $\d K$,
$$\int_{\partial K}f(\xx)\,d\vol_{N-1}(\xx)=
\int f({\bf x}(\phi)) \det G(\phi)^{1/2}\, d\phi_1\cdots d\phi_{N-1}\,,$$ can therefore  be rewritten as
\begin{equation}
\boxed{\,\int_{\partial K}f(\xx)\,d\vol_{N-1}(\xx) = \int_{\mathbb{S}^{N-1}} f({\bf x}_{\d K}({\bf n})) \,  { \det \tilde{C}(K,\nn)}\, d\sigma_{N-1}({\bf n}) \,.}
\label{intM}
\end{equation}
This equation will be used together with Steiner's formula in Section~\ref{s-steiner} to compute the volumes of Minkowski sums of ellipsoids.

\section{Volume Bounds Using Bounding Ellipsoids}\label{s-bounds}

We shall use the notation
\beq E_A  := \{\xx\in\IR^N: \xx^T A^{-2}\xx<1\}\,,\label{impl1}\eeq when $A$ is a symmetric positive-definite $N\times N$ matrix; equivalently, the  boundary of $E_A$ is given implicitly by \eqref{para1} or \eqref{x-from-n}.

Given inner and outer ellipsoidal bounds of the form
$$ E_{A_{\inner}} \,\subseteq\, \sum_i E_i \,\subseteq\, E_{A_{\Outer}}  \,, $$
and noting that $\vol(E_A)=\vol(\mathbb{B}^N)\det A$, where  $\Vol(\mathbb{B}^N)=\frac{\pi^{N/2}}{\Gamma(N/2 +1)}$ is the volume of the unit $N$-ball, we have the obvious volume bounds
\begin{equation}
 \Vol(\mathbb{B}^{N}) \det\left(A_{\inner}\right)  \,\leq\, \Vol\left(\sum_i E_i\right) \,\leq\, \Vol(\mathbb{B}^{N}) \det \left(A_{\Outer}\right).
\label{volbounds}
\end{equation}
In the following, we review some formulas for $A_\inner$ and $A_\Outer$ which we apply to \eqref{volbounds}, and which can be used to obtain further volume estimates in Section~\ref{s-steiner}.

\subsection{Optimal Lower Bounds for Minkowski Sums of Two Ellipsoids}\label{containment}

An ellipsoid can be fit inside the Minkowski sum $E_1+\cdots +E_m$ of the solid ellipsoids $E_j$ of Theorem  \ref{symmetric} by the following argument. Consider a solid ellipsoid, $E_\inner $, defined by  ${\bf x}^T A_\inner ^{-2} {\bf x} < 1$. Recalling \eqref{x-from-n}, we can  parameterize $\partial E_\inner $ by its normal $\nn$  as \beq\label{param} {\bf x}_{\partial E_\inner }({\bf n}) = \frac{A_\inner ^2 {\bf n}}{\|A_\inner  {\bf n}\|} \,, \eeq where $A_\inner $ is symmetric, positive-definite.

{The containment
 condition $E_\inner\subset E_1+\cdots+E_m$ can be written as the  inequality
\begin{equation} \nn^T\,
{\bf x}_{\partial E_\inner }({\bf n}) \leq \nn^T\,{\bf x}_{\partial (E_1  + \cdots + E_m)}({\bf n})\,,  \label{pos-hyp}
\end{equation}  for all $\nn\in\IS^{N-1}$.
By \eqref{new-mult} and \eqref{param}, we then obtain} the general condition
\beq\label{Einner}E_\inner \subseteq E_1+\cdots+ E_m \ \iff\   \|A_\inner \vv\|\le \sum_{j=1}^m\|A_j\vv\|,\ \forall  \vv\in\IR^N\,.\eeq
Hence by the triangle inequality, the matrix
\begin{equation}
 A_{\rm sum}    := \sum_{i=1}^{m} A_i
\label{Ainner}
\end{equation}
satisfies these conditions and hence $E_{A_{\rm sum}}\subseteq E_1+\cdots+E_m$.

When $E_\inner$ is contained in the Minkowski sum  $\Sigma := E_1+\cdots+E_m$, a boundary point $\xx_{\partial E_\inner }(\nn)\in\partial E_\inner $  is also in the boundary of $\Sigma$ if and only if equality holds in (\ref{pos-hyp}); i.e.,
\beq \label{innercontact} \|A_\inner \nn\|= \sum_{j=1}^m\|A_j\nn\|\,.\eeq
 If \eqref{innercontact} holds, then $\nn$ is also the unit normal to $\Sigma$ at $\xx(\nn)$.

For the case $m=2$, the inner ellipsoid $E_{\rm sum} := E_{A_{\rm sum}}$  will contact the boundary of the Minkowski sum at $2N$ (or more) points. Indeed, let $\vv_1,\dots,\vv_N$ be eigenvectors of $A_1^{-1}A_2$ (which is diagonalizable since it is a conjugate of $A_1^{-1/2}A_2A_1^{-1/2}$) with eigenvalues $\lambda_1,\dots,\lambda_N$, respectively. Then $A_2\vv_j=\lambda_jA_1\vv_j$, and hence by formula \eqref{new-mult} of Theorem~\ref{symmetric},
$$\xx_{\partial(E_1+E_2)}(\vv_j)=(1+\lambda_j)\frac{A_1^2 \vv_j}{\|A_1 \vv_j\|}=\frac{A_{\rm sum} ^2 \vv_j}{\|A_{\rm sum}  \vv_j\|}= \xx_{\partial E_{\rm sum} }(\vv_j)\,.$$ Therefore (when $m=2$), $E_{\rm sum} $ contacts the boundary of $E_1+E_2$ at the $2N$ points $\pm\xx_{\partial E_{\rm sum} }(\vv_j)$.
Having these $2N$ contacts, $E_{\rm sum} $ is a good lower bound for the Minkowski sum of two ellipsoids. However, except in special cases where  $E_1$ and $E_2$ have the same principal axes (i.e., $A_1$ commutes with $A_2$), $E_{\rm sum} $ will not coincide with the maximal volume inner ellipsoid described in Theorem~\ref{Jellipsoid} and  Lemma~\ref{Jlemma} below.

For the general case of the Minkowski sum $\Sigma$ of three or more solid  ellipsoids,  $\partial E_{\rm sum} \cap\partial\Sigma=\emptyset$.  Indeed, if none of the eigenvectors  of $A_1^{-1}A_2$ are eigenvectors of $A_1^{-1}A_3$, then $$\|A_{\rm sum} \nn\|< \sum_{j=1}^m\|A_j\nn\|$$ for all vectors $\nn$, and hence   $E_{\rm sum} $ can be dilated and remain inside $\Sigma$; Thus $E_{\rm sum} $ will not have maximal volume.  However, there will be contact points in special cases when the ellipsoids share the same semi-axes. (See Remark \ref{Jremark} below.

To determine if an inner ellipsoid has maximal volume, one can apply  the following result of F. John \cite{John48} (see also \cite{Ball}):

 \begin{theorem} {\rm (John \cite{John48})} Let $K$ be a convex body that is symmetric about $\bz$.  A solid ellipsoid $E=\{A\vv:\|\vv\|<1\}$ contained in  $K$ has maximal volume among all  ellipsoids contained in $K$ if and only if there exist points $\xx_1,\cdots, \xx_k$ ($k\ge N$) in $\partial E \cap\partial K$  and constants $c_1,\dots,c_k$ such that \beq\label{J}\sum_{j=1}^k c_j(\yy^T A^{-1}\xx_j)A^{-1}\xx_j=\yy\,,\eeq for all $\yy\in\IR^N$. Furthermore $E$ is unique.\label{John}\end{theorem}
 The solid ellipsoid $E$ is called the {\it L\"owner--John ellipsoid}. Note that the vectors $A^{-1}\xx_j$ lie in the unit sphere.

 \begin{example}Condition \eqref{J} holds  in the following cases: \begin{itemize}
 \item $k=N$,  the $A^{-1}\xx_j$ are orthonormal, and $c_j=1$ for $j=1,\dots, N$.

 \item $N=2,\ k\ge 3, \ A^{-1}\xx_j=[\cos(2\pi j/k),\sin (2\pi j/k)]^T$ and $c_j=2/k$ for $0\le j\le k-1$.

 \end{itemize}\end{example}

 \begin{remark}\label{Jremark} For example, if the $A_i$ are diagonal positive-definite matrices, then \eqref{innercontact} holds for $\nn=\ee_j$ (the standard basis vectors),     and the conditions in Theorem~\ref{John} are met with ${\bf x}_j=A_{\rm sum} \ee_j$.
\end{remark}

A formula for the maximal volume inner ellipsoid of the Minkowski sum $E_1+E_2$  was given by Chernousko \cite{Chernousko} (see also \cite{Kurzhanski1}). Chernousko's formula (equation \eqref{Ejohn2} below) can be described in terms of the operator geometric mean:

\begin{definition}  Let $P,\,Q$ be positive-definite symmetric matrices. The geometric mean $P\#Q$ of $P$ and $Q$ is given by $$P\#Q  := P^{1/2}\left(P^{-1/2}QP^{-1/2}\right)^{1/2}P^{1/2}\,.$$ \end{definition}

We note that $P\#Q=Q\#P$, and  $$P\#Q=P^{1/2}Q^{1/2}\ \iff\ PQ=QP\,.$$ (See \cite{Bhatia}.) The geometric mean $P\#Q$ can also be interpreted as the midpoint of the geodesic from $P$ to $Q$ in the Riemannian metric on the space of positive-definite matrices \cite[Th.~6.1.6]{Bhatia}.

\begin{theorem}\label{Jellipsoid}{\rm (Chernousko \cite{Chernousko})} Let $A,\,B$ be symmetric positive-definite $N\times N$ matrices, and let  $E_A,E_B$ be given by \eqref{impl1}. Then the L\"owner--John ellipsoid (maximal volume inner ellipsoid) $E_\John$ for the Minkowski sum $E_A+E_B$  is given by \beq\label{Ejohn2}E_\John=\{F(A,B)\,\vv:\|\vv\|<1\},\quad  F(A,B) := \left[A^2+2\,A^2\#B^2                    +B^2\right]^{1/2}.\eeq\end{theorem}

\begin{remark}
We  note that $$E_\John=E_{\rm sum}  \iff F(A,B)=A+B \iff AB=BA \iff (AB)=(AB)^T\,,$$ which occurs  when $E_A$ and $E_B$ share all their axes.\end{remark}

We provide here a short proof of Theorem \ref{Jellipsoid} using Theorem \ref{John}. First we describe $E_\John$ using a non-symmetric matrix:

\begin{lemma}\label{Jlemma} Let $A,\,B$ be symmetric positive-definite $N\times N$  matrices. Then the L\"owner--John ellipsoid  for the Minkowski sum $E_A+E_B$  is given by \beq\label{Sformula} E_\John=\{S\vv:\|\vv\|<1\},\quad S=A \,\left[\II+(A^{-1}B^2A^{-1})^{1/2}\right].\eeq
\end{lemma}

\begin{proof}  We first consider the case where $A=\II$. To show that $E_\John$ is defined by $A_{\rm sum} =\II+B$ in this case, we let $\nn_1,\dots,\nn_N$ be orthonormal eigenvectors of $B$ with eigenvalues $\lambda_1,\dots,\lambda_N$, respectively. Then $$
\|A_{\rm sum} \nn_j\|=1+\lambda_j= \|\II\nn_j\|+\|B\nn_j\|\,,$$ and thus by \eqref{innercontact},
 $${\bf x}_{\partial E_{\rm sum} }(\nn_j) = \frac{A_{\rm sum} ^2 {\nn_j}}{\|A_{\rm sum}  \nn_j\|}=(1+\lambda_j)\nn_j\in \partial E_{\rm sum} \cap\partial (E_A+E_B)\,,$$ where $E_{\rm sum} =\{A_{\rm sum} \uu:\|\uu\|<1\}$.  Then $A_{\rm sum} ^{-1}\,\xx_{\partial E_{\rm sum} }(\nn_j)=\nn_j$, and therefore John's condition \eqref{J} (with $\xx_j= {\bf x}_{\partial E_{\rm sum} }(\nn_j)$, $c_j=1$,  and $A$ replaced by $A_{\rm sum} $) is satisfied for the solid ellipsoid $E_{\rm sum} $.

 For the general case, we apply the linear transformation $A^{-1}$, so that we have
 $$ A^{-1}\cdot E_A=\IB^N,\quad A^{-1}\cdot E_B= \left\{\widetilde B\vv:\|\vv\|<1\right\},$$ where \beq\label{wB}\widetilde B :=\left[(A^{-1}B)(A^{-1}B)^T\right]^{1/2}= \left(A^{-1}B^2A^{-1}\right)^{1/2}\,.\eeq Since  $A^{-1}\cdot E_\John$ is the L\"owner--John ellipsoid for $A^{-1}\cdot E_A+A^{-1}\cdot E_B$, we have
 $$A^{-1}\cdot E_\John=\left\{(\II+\widetilde B)\vv:\|\vv\|<1\right\}.$$ Left-multiplying by $A$ and applying \eqref{wB}, we obtain \eqref{Sformula}.\end{proof}

 To complete the proof of Theorem~\ref{Jellipsoid}, we have by Lemma~\ref{Jlemma} and \eqref{impl}, $$E_\John=\{(SS^T)^{1/2}\vv:\|\vv\|<1\}\,,$$ where $S$ is given by \eqref{Sformula}.  We then have
 \beq\label{Ajohn1}  (SS^T)^{1/2} = \left[A^2+2A(A^{-1}B^2A^{-1})^{1/2}A+B^2\right]^{1/2}=F(A,B)\,.\eeq \qed

\begin{example}\label{example-John}

\begin{figure}[h]
  \centering
\includegraphics[width=6cm]{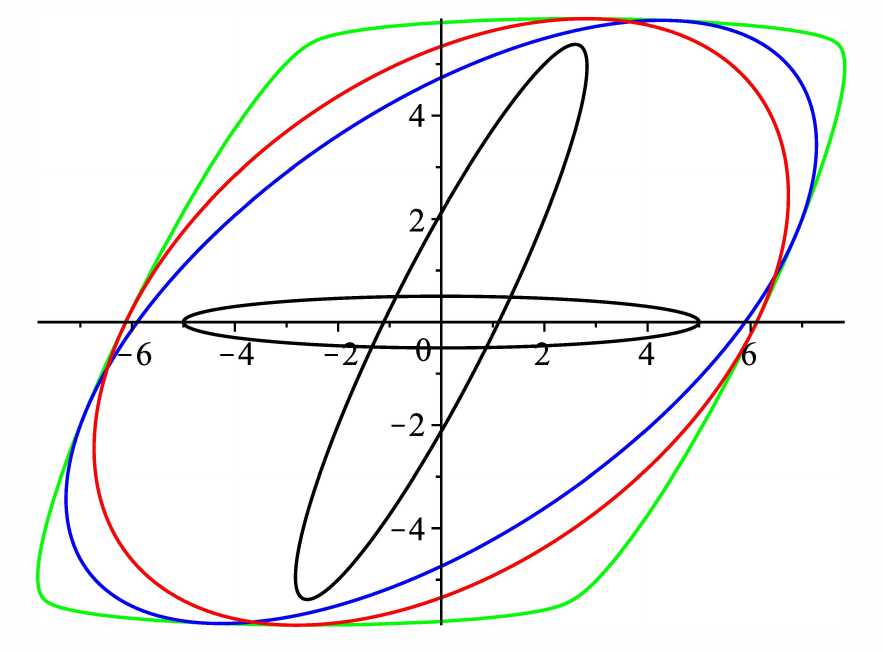}
\vskip 1.0 true in
\caption{Inner ellipses}
  \label{figure-John}
\end{figure}

Figure~\ref{figure-John} illustrates Theorem~\ref{Jellipsoid} using the matrices
$$A=\begin{pmatrix} 5&0\\0&\frac 12\end{pmatrix},\qquad B=\begin{pmatrix} 2&2\\2&5\end{pmatrix}.$$ In the figure, the ellipses $\d E_A,\,\d E_B\subset\IR^2$ are black, the Minkowski sum boundary $\d (E_A+E_B)$ is green, $\d E_{\rm sum} =\d E_{A+B}$ is blue, and $\d E_\John$ is red.

 The area of  $E_\John$ is approximately 113.14, whereas $E_{\rm sum} $ has area $\approx$\, 108.38.\end{example}

\begin{remark}  Kurzhanski--V\'alyi \cite{Kurzhanski1} gives a family of inner ellipsoids  $E_{\widehat S}$, with
\begin{equation*} \widehat S^2\,=\,
 S^{-1}\left[(S A^2 S)^{\frac{1}{2}} \,+\, (S B^2 S)^{\frac{1}{2}}\right]^2 S^{-1} \,,
\label{new-mult2}
\end{equation*}
where $S$ is any symmetric positive-definite matrix. The union of the ellipsoids $E_{\widehat S}$  covers the entire Minkowski sum $E_A+E_B$. If $S=A^{-1}$ or $S=B^{-1}$, then  $\widehat S =F(A,B)$.\end{remark}

\subsection{Comparison with Brunn-Minkowski}

The  Brunn-Minkowski inequality states that
$$ |{\rm \Vol}(K_1 + K_2)|^{1/N} \,\geq\, |{\rm \Vol}(K_1)|^{1/N} \,+\, |{\rm \Vol}(K_2)|^{1/N} \,. $$
(If $K_1,K_2$ are both convex and have positive volume, then equality holds if and only if they are homothetic, i.e., $K_1 = c K_2$.) For ellipsoidal bodies $E_1$ and $E_2$, the first inequality of \eqref{volbounds} gives a sharper inequality than Brunn-Minkowski:

\begin{theorem}\label{lowerbound} Let $E_i=A_i\cdot\IB^N$,  $i=1,2$, be ellipsoids, where $A_1,A_2$ are positive-definite symmetric matrices. Then
\begin{eqnarray*}
\left[ \Vol (E_1+E_2)\right]^{1/N}
&\ge& \Vol(\mathbb{B}^{N})^{1/N} \det\left(F(A_1,A_2)\right)^{1/N} \\
&\ge& \Vol(\mathbb{B}^{N})^{1/N} \det\left(A_1+A_2\right)^{1/N} \\
&\ge& \Vol(E_1)^{1/N}+\Vol(E_2)^{1/N}\,,
\end{eqnarray*}
{ where $F(A,B)$ is given by \eqref{Ejohn2} (or equivalently by \eqref{Ajohn1}).}

\end{theorem}
\begin{proof} The first two inequalities of the theorem follow from the optimality of John's ellipsoid and  \eqref{volbounds}.
The Minkowski inequality for determinants yields
\begin{eqnarray*}\Vol(\mathbb{B}^{N})^{1/N} \det\left( A_1+A_2\right)^{1/N}   &\ge &  \Vol(\mathbb{B}^{N})^{1/N}\left[(\det A_1)^{1/N}+(\det A_2)^{1/N}\right]\\&=&\Vol(E_1)^{1/N}+\Vol(E_2)^{1/N}\,.\end{eqnarray*}\end{proof}

Consider for example, the case when $N=2$ and $E_1$, $E_2$ are  degenerate ellipses with $A_1 = {\rm diag}[a_1,0]$ and
$A_2 = {\rm diag}[0,a_2]$. Each has zero area, but the Minkowski sum will be a $(2 a_1) \times (2 a_2)$ rectangle. In this
case, Brunn-Minkowski gives $4 a_1  a_2 > 0$ while Theorem~\ref{lowerbound}  gives $4a_1a_2>\pi a_1 a_2$.

Theorem \ref{lowerbound} generalizes to the sum of $m$ ellipsoids, although the bounds become looser.
For example,
given three ellipsoids $E_1,\,E_2,\,E_3$, defined by $A_1$, $A_2$, $A_3$, a lower bound on the Minkowski sum determinant can be obtained recursively by considering the three positive matrices
$$ A' = F(F(A_1,A_2),A_3),\  A'' = F(F(A_1,A_3),A_2),\ A''' = F(F(A_3,A_2),A_1), $$ each defining an inner ellipsoid for $E_1+E_2+E_3$. Then
\begin{eqnarray*}
\left[ \Vol (E_1+E_2+E_3)\right]^{1/N}
&\ge& \Vol(\mathbb{B}^{N})^{1/N} \max\{\det A',\det A'',\det A'''\}^{1/N} \\
&\ge& \Vol(\mathbb{B}^{N})^{1/N} \det\left(A_1+A_2+A_3\right)^{1/N} \\
&\ge& \Vol(E_1)^{1/N}+\Vol(E_2)^{1/N}+\Vol(E_3)^{1/N}\,.
\end{eqnarray*}

\subsection{Upper Bounds for Minkowski Sums of Ellipsoids}\label{upper}

A family of solid ellipsoids containing the Minkowski sum $E_1+\cdots +E_m$ is given as follows
(see \cite{Durieu,Kurzhanski1,Sholokhov}):

\begin{equation} E_\Outer  ^\gamma=\left\{A_\gamma\vv:\|\vv\|<1\right\}\,,\quad
A^2_\gamma \, :=\, \sum_{i=1}^{m} \gamma_i A_i^2\,, \quad E_i=A_i\cdot\IB^N\,,
\label{genbound1}
\end{equation}
for $\gamma=(\gamma_1,\dots,\gamma_m)$ with
\begin{equation}
\gamma_i >  0 \,\,\, {\rm and} \,\,\,
\sum_{i=1}^{m}  \frac{1}{\gamma_i} = 1\,.
\label{gamma1}
\end{equation}

By the method of Section~\ref{containment},
 $E_1+\cdots +E_m\subseteq E^\gamma_\Outer  $ if and only if
\beq\label{Agamma} \sum_{i=1}^{m} \| A_i {\bf u} \| \leq \|A_\gamma\uu\|,\quad \forall \uu\in\IR^N \,. \eeq
In fact, if $\gamma_i{>0}$ for $i=1,...,m$, then by the Cauchy-Schwarz inequality,
$$ \sum_{i=1}^{m} \| A_i {\bf u} \| \,=\, \sum_{i=1}^{m} \gamma_i^{-1/2} \| \gamma_i^{1/2} A_i {\bf u} \|
\,\leq\, \left(\sum_{i=1}^{m} \gamma_i^{-1}\right)^{\frac{1}{2}} \left(\sum_{i=1}^{m} \| \gamma_i^{1/2} A_i {\bf u} \|^2\right)^{\frac{1}{2}} \,. $$
Imposing the constraint
$$ \sum_{i=1}^{m} \gamma_i^{-1} = 1 $$
and observing that
$$ \left(\sum_{i=1}^{m} \| \gamma_i^{1/2} A_i {\bf u} \|^2\right)^{\frac{1}{2}} =
 \left(\sum_{i=1}^{m} {\bf u}^T(\gamma_i A_i^2) {\bf u} \right)^{\frac{1}{2}} =
 \left({\bf u}^T A^2_\gamma  {\bf u} \right)^{\frac{1}{2}} = \|A_\gamma \uu\|$$
  reproduces the well-known constraint \eqref{Agamma}, and hence $E_1+\cdots +E_m\subset E^\gamma_\Outer  $ when $\ga$ satisfies \eqref{gamma1}.

In the case when $m=2$, the outer ellipsoid defined by \eqref{genbound1}--\eqref{gamma1} has minimal volume when \beq\gamma_1=1+\beta^{-1}\,,\quad  \gamma_2=1+\beta\,,\label{gammabeta}\eeq where $\beta$ is the (unique) positive solution of the equation \beq\sum_{j=1}^N \frac {1-\beta^2\,\lambda_j^2(A_1^{-1}A_2)}{1+\beta\,\lambda_j^2(A_1^{-1}A_2)}=0\;;\label{beta}\eeq see \cite{Halder, Kurzhanski1,Sholokhov}. Thus,  \eqref{volbounds} provides the upper bound on the volume
\beq\label{volupper} \vol(E_1+E_2) \le \vol(\IB^N)\det(\ga_1A_1^2+\ga_2A_2)^{1/2}\,,\eeq
where $\ga_1,\ga_2$ are given by \eqref{gammabeta}--\eqref{beta}. When $A_2 = RA_1 R^T$, with $R$ being a rotation matrix, symmetry yields $\beta=1$ and thus $\gamma_1 = \gamma_2 = 2$.

In some contexts where rapid computations are required, alternative choices can be made with good effect such as
\begin{equation} \beta'=\sqrt{\frac{{\rm tr}(A_1^2)}{{\rm tr}(A_2^2)}}
\ \Longrightarrow\
\gamma'_1 = 1 + \sqrt{\frac{{\rm tr}(A_2^2)}{{\rm tr}(A_1^2)}} \,\,\,\, {\rm and} \,\,\,\,\gamma'_2 = 1 + \sqrt{\frac{{\rm tr}(A_1^2)}{{\rm tr}(A_2^2)}}.
\label{heuristic1}
\end{equation}
The choice in (\ref{heuristic1})  gives the enclosing ellipsoid
that minimizes the sum of squared semi-axes lengths \cite[Lemma~2.5.2]{Kurzhanski1}.
A heuristic choice motivated by \eqref{heuristic1}  for the case of an $m$-fold Minkowski sum is
\begin{equation} \gamma'_i = \frac{\sum_{j=1}^{m} \sqrt{{\rm tr}(A_j^2)}}{\sqrt{{\rm tr}(A_i^2)}} \,.
\label{heuristic222}
\end{equation}

It can be shown that the minimal volume ellipsoid of the form \eqref{genbound1} enclosing an $m$-fold Minkowski sum is defined by \cite{Halder,Kurzhanski1}
\begin{equation}
A({\bf l}) \,=\, \left(\sum_{j=1}^{m} \|A_j {\bf l}\|\right)^{\frac{1}{2}} \left(\sum_{i=1}^{m} \frac{A_i^2}{\|A_i {\bf l}\|}\right)^{\frac{1}{2}}
\label{minvolm0}
\end{equation}
where ${\bf l}$ is the unit vector  ${\bf u} \in \mathbb{S}^{N-1}$ that
minimizes $\det A({\bf u})$. That is,  the optimal choice for $\gamma_i$ is defined by
\beq\label{thegammas} \frac{1}{\gamma_i} \, :=\, \frac{\|A_i {\bf l}\|}{\sum_{j=1}^{m} \|A_j {\bf l}\|}\,. \eeq
An algorithm for finding ${\bf l}$ so that $A({\bf l})$ defines the minimal volume ellipsoid
is given in \cite{Halder}.

Thus we have the upper bound \beq\label{upperbound1} \vol_N(E_1+\cdots+E_m) \le \vol(\IB^N)\det A({\bf l})\,,\eeq where $A({\bf l})$ is given by \eqref{minvolm0}.
Applying the isoperimetric inequality of Theoorem~\ref{iso1+t}, we obtain the following upper bound:
\begin{theorem}\label{upperboundt} Let $\Sigma=E_1+\dots+ E_m$, where $E_1,\dots,E_m$ are ellipsoids in $\IR^N$ given as in Theorem~\ref{symmetric}. Then
\beq\label{upperbound2}\vol_N(\Sigma)\le \frac{[\det  A({\bf l})]^{1/N}}N\,\vol_{N-1}(\d\Sigma)\,,\eeq
 where $A({\bf l})$ is given by \eqref{minvolm0}.\end{theorem}

\section{Bounds on Volume Using Steiner's Formula}\label{s-steiner}

Given a convex body $K \subset \mathbb{R}^N$,
{\it Steiner's Formula} gives the volume of the offset body as
\begin{equation}\label{steiner}
\vol(K+ r \mathbb{B}^N) = \sum_{j=0}^{N} \left(\begin{array}{c}
N \\
j \end{array}\right) W_j(K) r^j \,,\end{equation}
where  the quantities $W_j(K)$ are the {\it quermasssintegrals} of $K$.  (See \cite{BZ, Gruber, Kurzhanski1}.) In particular,
Steiner's Formula for the area of the Minkowski sum of a 2D convex body $K$ with a disk of
radius $r$ is
$$ {\mathcal A}(K+ r\mathbb{B}^2) \,=\, {\mathcal A}(K) + r {\mathcal L}(\partial K) + \frac{r^2}{2} {\mathcal K}(\partial K) $$
where ${\mathcal A}(K)$ is the area of $K$, ${\mathcal L}(\partial K)$ is the length of the boundary (i.e., its perimeter), and
$$ {\mathcal K}(\partial K) = 2\pi $$
is the integral of curvature around the boundary. Consequently,
$$ {\mathcal A}(K+ r\mathbb{B}^2) \,=\, {\mathcal A}(K) + r {\mathcal L}(\partial K) +
{\mathcal A}(r\mathbb{B}^2) \,.$$
Since
$$ {\mathcal A}(E_{A_1} + E_{A_2}) = (\det A_2) {\mathcal A}(A_{2}^{-1} E_{A_1} + \mathbb{B}^2)\,, $$
Steiner's formula can be used to compute ${\mathcal A}(E_{A_1} + E_{A_2})$ exactly as
\begin{multline} {\mathcal A}(E_{A_1} + E_{A_2}) = {\mathcal A}(E_{A_1}) + (\det A_2) {\mathcal L}(\partial(A_{2}^{-1} E_{A_1})) +
 {\mathcal A}(E_{A_2}) \\  = \pi(\det A_1+\det A_2) +4\lambda_1 {(\det A_2)} {\bf E}\left([1-\lambda_2^2/\lambda_1^2]^{1/2}\right),\label{2sum}\end{multline} where $\lambda_1,\lambda_2$ are the eigenvalues of $A_{2}^{-1}{A_1}$ and ${\bf E}(x) := \int_0^{\pi/2} \sqrt{1-x^2\sin^2\theta}\,d\theta$ is an elliptic integral.

Letting
$K = A_{3}^{-1}(E_{A_1} + E_{A_2})$, it is possible to repeat the same procedure as above to compute
the exact area ${\mathcal A}(E_{A_1} + E_{A_2} + E_{A_3})$ as
$$ {\mathcal A}(E_{A_1} + E_{A_2} + E_{A_3}) = {\mathcal A}(E_{A_1} + E_{A_2} ) + (\det A_3) {\mathcal L}(\partial(A_{3}^{-1} E_{A_1} + A_{3}^{-1} E_{A_2})) +
 {\mathcal A}(E_{A_3}) \,. $$
This is exactly computable  using \eqref{2sum} and  the general fact that
$$ {\mathcal L}(\partial(K_1+K_2)) = {\mathcal L}(\partial K_1) + {\mathcal L}(\partial K_2) \,. $$
This approach leads to a recursive algorithm to exactly compute the area of an $m$-fold Minkowski sum of
ellipses in the plane.

In higher dimensions, we can take a similar approach to tightly bound the volume of Minkowski sums, but
the approach will no longer give an exact equality for $m >2$ ellipsoids. This is now demonstrated in the 3D case.

Steiner's Formula for the volume of the Minkowski sum of a 3D convex body with a ball of
radius $r$ is
$$ \vol (K+ r\mathbb{B}^3) \,=\, \vol (K) + r {\mathcal A}(\partial K) + {r^2} {\mathcal M}(\partial K) +
\frac{r^3}{3} {\mathcal K}(\partial K) $$
where
$\vol (K)$ is the volume of $K$, ${\mathcal A}(\partial K)$ is the surface area of the boundary $\partial K$, and
$ {\mathcal M}(\partial K)$ and $ {\mathcal K}(\partial K)$
are respectively the integral of mean and Gaussian curvature over the whole boundary $\partial K$.
From the Gauss-Bonnet Theorem,
$$ {\mathcal K}(\partial K) = 2\pi \chi(\partial K) = 4\pi. $$
Consequently
$$ \vol (K+ r\mathbb{B}^3) \,=\, \vol (K) + r {\mathcal A}(\partial K) + {r^2} {\mathcal M}(\partial K) + \vol (r\mathbb{B}^3) \,.$$
Therefore, for the sum of a pair of 3D ellipsoids
$$ \vol (E_{A_1} + E_{A_2}) = (\det A_2) \vol (A_{2}^{-1} E_{A_1} + \mathbb{B}^2)\,, $$
the same approach as in the planar case can be used to exactly compute
\begin{multline*} \vol (E_{A_1} + E_{A_2}) \,=\, \vol (E_{A_1}) +
(\det A_2) {\mathcal A}(\partial (A_2^{-1} E_{A_1}))\\ + (\det A_2) {\mathcal M}(\partial (A_2^{-1} E_{A_1})) + \vol (E_{A_2}) \,. \end{multline*}
The quantities in this formula can be computed using \eqref{intM}.
In particular, for an ellipsoid $E$,
$$ {\mathcal A}(\partial E) \,=\, \int_{{\bf n} \in \mathbb{S}^{2}} \det \tilde{C}_E({\bf n}) \, d\sigma_2({\bf n})\,.  $$
 Furthermore, by Theorem~\ref{curvatures},
\begin{multline} {\mathcal M}(\partial E_{A})\,=\, \frac{1}{2} \int_{\mathbb{S}^{2}} {\rm tr}\left[\tilde{C}_E({\bf n})^{-1}\right] \, \det \tilde{C}_E({\bf n}) \, d\sigma_2({\bf n})\\ =  \frac{1}{2} \int_{ \mathbb{S}^{2}} {\rm tr}[\tilde{C}_E({\bf n})]  \, d\sigma_2({\bf n})\,=\, \frac{1}{2} \int_{ \mathbb{S}^{2}} {\rm tr}[{C}_E({\bf n})]  \, d\sigma_2({\bf n})  \,. \label{meancurv}\end{multline}
If it is assumed that ${\mathcal A}$ and ${\mathcal M}$ are exactly computable for 3D ellipsoids, then the above
provides a closed-form formula for $\vol (E_{A_1} + E_{A_2})$.

Following the same logic as in the planar case,
\begin{multline}\label{3sum} \vol (E_{A_1} + E_{A_2} + E_{A_3}) =\vol (E_{A_1} + E_{A_2}) +
(\det A_3) {\mathcal A}(\partial (A_3^{-1} E_{A_1} + A_3^{-1} E_{A_2}))\\ +  (\det A_3) {\mathcal M}(\partial (A_3^{-1} E_{A_1}+A_3^{-1} E_{A_2})) + \vol (E_{A_3}) \,, \end{multline}
where $\vol (E_{A_1} + E_{A_2})$ is fed forward from the previous step.
Moreover,  by the additivity of $W_2$ in 3D (see \cite{Gruber,Schneider}, or by \eqref{Csum} and  \eqref{meancurv}),
$$ {\mathcal M}(\partial(K_1+K_2)) = {\mathcal M}(\partial K_1) + {\mathcal M}(\partial K_2) \,, $$
but no such equality exists for surface area. {Since it is known that $\acal(\partial K)=W_1(K)$ is the average area of the orthogonal projections of $K$ onto planes in $\IR^3$ \cite{Schneider}}, it follows that if $E_\inner  \subseteq K \subseteq E_\Outer$ then
\beq\label{EKE}\acal(\d E_\inner ) \leq \acal(\d K) \leq \acal(\d E_\Outer). \eeq
This can  be used together with the ellipsoidal bounds in Section~\ref{s-bounds}  to provide volume bounds in \eqref{3sum}. The same reasoning can be applied in higher dimensions using Steiner's Formula \eqref{steiner}, where the result of Theorem~\ref{curvatures} continues to be applicable.



\section*{Acknowledgements} This work was supported by National Science Foundation
grants  CCF-1640970 and IIS-1619050.


\begin{thebibliography}{99}

\bibitem{Ball91} K. M. Ball, Volume ratios and a reverse isoperimetric inequality, J. London Math.\ Soc. (2) 44 (1991), 351--359.

\bibitem{Ball}K. M. Ball,  Ellipsoids of maximal volume in convex bodies, Geom.\ Dedicata 41 (1992),
241--250.

\bibitem{Bhatia} R. Bhatia, Positive Definite Matrices, Princeton University Press, 2006.

\bibitem{BZ} Yu. D. Burago and V. A.  Zalgaller, Geometric inequalities. Translated from the Russian by A. B. Sosinski\u i. Grundlehren der Mathematischen Wissenschaften [Fundamental Principles of Mathematical Sciences], 285. Springer Series in Soviet Mathematics. Springer-Verlag, Berlin, 1988.

\bibitem{Chernousko} F. L. Chernousko, State Estimation for Dynamic Systems, CRC Press, 1994.


\bibitem{cs} G. S. Chirikjian and B. Shiffman, Collision-free configuration-spaces in macromolecular crystals, Robotica 34 (8), 1679--1704 (2016).

\bibitem{DeG} E. De Giorgi, Su una teoria generale della misura $(r - 1)$-dimensionale in uno spazio
ad $r$ dimensioni, Ann.\ Mat. Pura Appl. \ 36 (1954), 191--213.

\bibitem{Durieu} C. Durieu, E. Walter, and B. Polyak, Multi-Input Multi-Output Ellipsoidal State Bounding, Journal of Optimization Theory and Applications, Vol. 111, No. 2, pp. 273--303, 2001.

\bibitem{Federer} H. Federer, Geometric Measure Theory, Springer-Verlag, Berlin, 1969.

\bibitem{Fleming} W. H. Fleming, Early Developments in
Geometric Measure Theory, Indiana Univ.\ Math.\ J., Vol. 69, No. 1 (2020), 5--36.

\bibitem{Firey} W. J. Firey,
The determination of convex bodies from their mean radius of curvature functions.
Mathematika 14 (1967), 1--13.

\bibitem{gardner}
R. J. Gardner, The Brunn-Minkowski inequality. Bull.\ Amer.\ Math.\ Soc.
39, 355--405, 2002.

\bibitem{grinberg}
E. L. Grinberg. Isoperimetric inequalities and identities for k-dimensional crosssections
of convex bodies. Math.\ Ann. 291. 75--86, 1991.

\bibitem{Gruber} P. M. Gruber, Convex and
Discrete Geometry, Springer, 2007.

\bibitem{Halder} A.
Halder,  On the parameterized computation of minimum volume outer ellipsoid of Minkowski sum of ellipsoids. In 2018 IEEE Conference on Decision and Control (CDC) (pp. 4040--4045). IEEE, 2018.

\bibitem{John48} F. John, Extremum problems with inequalities as subsidiary conditions. Studies and Essays Presented to R. Courant on his 60th Birthday, January 8, 1948, Interscience Publishers, Inc., New York, N. Y., 1948, pp.\  187--204.

\bibitem{Kurzhanski1}
A. B. Kurzhanski and I. V\'{a}lyi, Ellipsoidal Calculus for Estimation
and Control, Systems \& Control: Foundations and Applications,
Birkh\"{a}user Boston and International Institute for Applied Systems
Analysis, 1997.

\bibitem{Kurzhanski2} A. A. Kurzhanskiy and P. Varaiya, Ellipsoidal Toolbox (ET), 45th IEEE Conference on Decision and Control, pp. 1498--1503, 2006.



\bibitem{ruan2020} S. Ruan and G. S. Chirikjian, Closed-Form Minkowski Sums of Convex Bodies with
Smooth Positively Curved Boundaries,
preprint, arxiv.org/abs/2012.15461.

\bibitem{ruan2019} S. Ruan, K. L. Poblete, Y. Li, Q. Lin, Q. Ma, G. S. Chirikjian,
Efficient Exact Collision Detection between Ellipsoids and Superquadrics via Closed-form Minkowski Sums.
2019 International Conference on Robotics and Automation (ICRA), 1765--1771.

\bibitem{santalo} L. A. Santal\'o, On complete systems of inequalities between elements of a plane convex figure. (Spanish) Math. Notae 17 (1959/61), 82--104.

\bibitem{Schneider}  R. Schneider, Convex Bodies: The Brunn-Minkowski Theory, 2nd ed. Encyclopedia of Mathematics and Its Applications 151, Cambridge Univ.\ Press, 2014.




\bibitem{SLC} B. Shiffman, S. Lyu, G. S. Chirikjian, Mathematical aspects of molecular replacement. V. Isolating
feasible regions in motion spaces, Acta Cryst.  A76, 145--162, 2020.

\bibitem{Sholokhov}
O. V. Sholokhov,  Minimum-volume ellipsoidal approximation of the sum of two ellipsoids. Cybernetics and Systems Analysis, 47(6), pp.\ 954--960, 2011.

\bibitem{Stolzenberg}G. Stolzenberg, Volumes, limits, and extensions of analytic varieties. Lecture Notes in Mathematics, No.\ 19, Springer-Verlag, Berlin-New York, 1966.


\bibitem{W} H. Whitney, Geometric integration theory. Princeton University Press, Princeton, N. J., 1957.

\bibitem{yy}
Y. Yan  and G..S. Chirikjian, Closed-form characterization of the Minkowski sum and difference of two ellipsoids. Geometriae Dedicata, 177(1), pp.\ 103--128, 2015




\end{thebibliography}
\end{document}